\documentclass[12pt,draft]{amsart}
\usepackage{multido,pstricks,pst-plot,pstricks-add}
\usepackage[all]{xy}
\usepackage{amsfonts}
\usepackage{amssymb}

\usepackage{enumerate}

\renewcommand{\ge}{\geqslant}

\newtheorem{teo}{Theorem}[section]
\newtheorem{lem}[teo]{Lemma}

\theoremstyle{definition}
\newtheorem{dfn}[teo]{Definition}
\newtheorem{rk}[teo]{Remark}
\newtheorem{ex}[teo]{Example}

\def\<{\langle}
\def\>{\rangle}
\def\ss{\subset}

\def\g{\gamma}

\def\t{\tau}

\def\f{{\varphi}}

\def\G{{\Gamma}}

\def\Z{{\mathbb Z}}

\def\cT{{\mathcal T}}

\def\Aut{\operatorname{Aut}}
\def\Out{\operatorname{Out}}

\def\Id{\operatorname{Id}}

\def\1{\mathbf 1}

\def\Rist{\operatorname{Rist}}
\def\St{\operatorname{St}}
\def\Orb{\operatorname{Orb}}

\begin{document}

\title[Reidemeister classes in weakly branch groups]
{Reidemeister classes in some weakly branch groups}

\author{Evgenij Troitsky}
\thanks{This work is
supported by the Russian Science Foundation under grant 16-11-10018.}
\address{Dept. of Mech. and Math., Moscow State University,
119991 GSP-1  Moscow, Russia}
\email{troitsky@mech.math.msu.su}
\urladdr{
http://mech.math.msu.su/\~{}troitsky}

\keywords{Reidemeister number, $R_\infty$-group, twisted conjugacy
class, residually finite group, weakly branch group}
\subjclass[2000]{20E45; 
20B35;  	
20F28;  
20F65  
}

\begin{abstract}
We prove that a saturated weakly branch group $G$ has the
property $R_\infty$ (any automorphism $\phi:G\to G$
has infinite Reidemeister number) in each of the following cases: 

1) any element of $\Out(G)$ has finite order;

2)  for any $\phi$ the number of orbits on levels of the tree automorphism $t$ inducing
$\phi$ is uniformly bounded
and $G$ is weakly stabilizer transitive;

3) $G$ is finitely generated, prime-branching, and 
weakly stabilizer transitive
with some non-abelian stabilizers
(with no restrictions on automorphisms).

Some related facts and generalizations are proved.
\end{abstract}

\maketitle

\section*{Introduction}
Consider an automorphism $\phi:G\to G$ of a (countable discrete) group. The \emph{Reidemeister number} $R(\phi)$ is the
number of its \emph{Reidemeister} or \emph{twisted conjugacy classes},
i.e. the classes of the following equivalence relation:
$g\sim hg\phi(h^{-1})$, $h,g\in G$. The
Reidemeister class of an element $g$ we denote by $\{g\}_\phi$.

A group has the $R_\infty$ property if $R(\phi)=\infty$ for any automorphism $\phi:G\to G$.
The problem of determining of groups having the $R_\infty$ property 
was raised by A.Fel'shtyn and co-authors in relation with an older conjecture by A.Fel'shtyn and R.Hill \cite{FelHill}: $R(\phi)$ is equal to the number of fixed points of the associated homeomorphism $\widehat{\phi}$ of the unitary
dual $\widehat{G}$, if one of these numbers is finite. 
This conjecture is called TBFT (twisted Burnside-Frobenius
theorem), because it generalizes to infinite groups and to the
twisted case the classical Burnside-Frobenius theorem: the number
of conjugacy classes of a finite group is equal to the number
of equivalence classes of its irreducible representations.
The question about TBFT formally has a positive answer
for $R_\infty$ groups. So, the $R_\infty$ problem is in some sense
complementary to the TBFT.

The TBFT conjecture was proved for finite, abelian and 
abelian-by-finite groups \cite{FelHill}. 
The further development, examples, counterexamples and
modifications can be found in \cite{FelTroVer,polyc,FelshtynTroitskyFaces2015JGT,%
ncrmkwb,TroTwoEx,TroLamp}.

The property $R_\infty$ was proved and disproved for many groups
and the number of papers on the subject and related questions
is too large to list all
of them and we restrict
ourselves to giving reference to several papers and bibliography
overview therein: \cite{TabWong,
bfg,
MubeenaSankaran2014Canad,Nasybull2012,
gowon1,
GoWon09Crelle,Romankov,
Jabara,
DekimpeGoncalves2014BLMS,
FelNasy2016JGT}.
Dynamical aspects  
of Reidemeister numbers are discussed in \cite{FelshB}.
Some direct topological consequences of the property $R_\infty$ for Jiang-type spaces are discussed in \cite{GoWon09Crelle}.

In \cite{FelLeonTro} the $R_\infty$ property was proved for a wide
class of saturated weakly branch groups. 

In the present paper we develop these results and prove
in Theorem \ref{teo:finordaut} that if any automorphism of a saturated weakly branch group $G$ is a composition of an inner automorphism and of a finite
order automorphism, then $G$ has the property $R_\infty$.
In particular, this theorem holds for the Grigorchuk group and
for the Gupta-Sidki group. In some specific cases the result can
be obtained from \cite{Jabara}.

We introduce the property WST (Definition \ref{dfn:trgroup}) and prove
that if for any automorphism $\phi$ of a 
saturated weakly branch
WST group $G$, induced by
an automorphism $t$ of the tree, i.e. $\phi(g)=tgt^{-1}$,
restrictions of $t$ on levels have a uniformly bounded number
of orbits, then $G$ has the property $R_\infty$ (Theorem 
\ref{teo:finnumoforb}).

In Theorem \ref{teo:primecase} we prove
 the $R_\infty$ property without any
restrictions on the structure of the automorphism group of
a finitely generated saturated weakly branch WST group $G$,
but with the restriction on branching numbers to be prime
and with an additional restriction on stabilizers. 

We prove that a saturated weakly branch
group on a spherically symmetric tree, such that 
any level stabilizer contains an odd permutation at some level,
is an $R_\infty$ group
(Theorem \ref{teo:gengen}).

\medskip
\textsc{Acknowledgement:} 
The author is indebted to A.~Fel'shtyn and the MPIM for helpful 
discussions in the Max-Planck Institute for
Mathematics (Bonn) in February, 2017, 
to V.~Manuilov for useful suggestions, and to A.~Jaikin-Zapirain 
for a bibliography reference.

This work is
supported by the Russian Science Foundation under grant 16-11-10018.

\section{Preliminaries}\label{sec:prelimi}
First, we recall some necessary facts about Reidemeister
classes.

\begin{lem}\label{lem:phiorbitsinclasses}
Any Reidemeister class of $\phi$ is formed by some $\phi$-orbits.
\end{lem}

\begin{proof}
Indeed, $\phi(g)=g^{-1}g\phi(g)$.
\end{proof}

\begin{dfn}
Denote by $\t_g$ the \emph{inner automorphism}: $\t_g(x)=gxg^{-1}$.
\end{dfn}

From the equality
$$
xy\f(x^{-1})g=x (yg) g^{-1}\f(x^{-1})g=x (yg)(\t_{g^{-1}}\circ\f)(x^{-1})
$$
we immediately obtain the following statement.

\begin{lem}\label{lem:ReidInnerShifts}
A right shift by $g\in G$ maps Reidemeister classes of $\phi$
onto Reidemeister classes of $\t_{g^{-1}}\circ \f$, 
In particular, $R(\t_g\circ \phi)=R(\phi)$.
\end{lem} 

\begin{lem}[{\cite[Prop.~3.4]{FelLuchTro}}]\label{lem:jaforfg}
Suppose, $\phi$ is an automorphism of a finitely 
generated residually finite group.
Let $R(\phi)=r<\infty$. Then the number of fixed elements
of $\phi$ is bounded by a function depending only on $r$.
\end{lem}


Now we pass to groups acting on trees and give
some known and new definitions and facts.

Let $\cT$ be a spherically symmetric rooted tree.
This means that all vertexes of the same level have the same
number of immediate descendants (\emph{branching index}).

Denote by $D(v)$ the set of immediate descendants of
a vertex $v\in \cT$.

A group $G$ acting faithfully on a rooted tree is 
a \emph{weakly branch group}, if for any vertex $v$ of $\cT$,
there exists an element of $G$ which acts nontrivially on the
subtree $\cT_v$ with the root vertex $v$ and trivially outside
this subtree. In other words, the \emph{rigid stabilizer}
$\Rist_v$ of any vertex $v$ is non-trivial.

Evidently a faithful tree group is residually finite.

We will denote by $\St(v)$ the \emph{stabilizer of a vertex}
$v\in \cT$; and by $\St_j$ the \emph{stabilizer of level} $L_j$,
i.e. $\St_j=\cap_{v\in L_j} \St(v)$.

A group $G$ is  \emph{saturated} if, for every
positive integer $n$, there exists a characteristic subgroup
$H_n \ss G$ acting trivially on the $n$-th level of $\cT$ and
level transitive on any subtree $\cT_v$ with $v$ in the $n$-th
level.

\begin{teo}[\cite{LavrNekr}]\label{teo:LavrNekr}
Suppose, $G$ is a saturated weakly branch group on a tree $\cT$. 
Then its automorphism group $\Aut G$ coincides with the normalizer of $G$ in the full group of isometries $Iso(\cT)$ of the rooted tree $\cT$: every automorphism $\phi$ of the group $G$ is induced by the conjugation by an element $t$
from the normalizer and the centralizer of $G$ in $Iso(\cT)$ is trivial.
\end{teo}


\begin{dfn}\label{dfn:subgrdes}
For  a group $G$ acting on $\cT$ and
any vertex $v\in\cT$ denote by $G_{\{v\}}$
the subgroup of all elements $g\in G$
fixing $v$ and all vertexes of $\cT$ from the next level,
except of immediate descendants of $v$.

In other words, if $v\in L_j$, then
$$
G_{\{v\}}=\bigcap_{w\in L_j, \: w\ne v} \quad \bigcap_{u \in D(w)}
\St(u)
$$
Thus,
$$
\Rist_v \ss G_{\{v\}} \ss \St_j.
$$
\end{dfn}

\begin{dfn}\label{dfn:trgroup}
We call a group $G$ acting on $\cT$ \emph{weakly stabilizer
transitive} (WST) if for any vertex $v$ one can find a vertex
$v_0 \in \cT_v$ such that $G_{\{v_0\}}$ acts
transitively on immediate descendants of $v_0$.
\end{dfn}

\begin{rk}\label{rk:Gvforleveltrans}
If $G$ acts level-transitively, then $G_{\{v\}}$ are pairwise
isomorphic for $v$ from the same level. Also, they pairwise
commute and we can introduce the following well-defined group
$\G_{\{i\}}$.  
\end{rk}

\begin{dfn}\label{dfn:gammai}
Denote
$$
\G_{\{i\}}:= \prod_{v\in L_{i-1}} p_i( G_{\{v\}}), 
$$
where $p_i: G\to G/\St_i$ is the natural projection.
\end{dfn}

Let $t$ be an automorphism of a tree $\cT$. Let $\Orb_i(t)$ be
the number of orbits of $t$ at the level $L_i$. Evidently,
\begin{enumerate}[1)]
\item $\Orb_i(t)$ is a not-decreasing function of $i$;
\item a fixed vertex of $t$ may be only a successor of a fixed vertex;
\item if there is a fixed vertex at the level $i+1$, then
$\Orb_{i+1}(t)>\Orb_i(t)$.
\end{enumerate}

So, we have two possibilities:
\begin{enumerate}[(a)]
\item $\Orb_i(t)\longrightarrow \infty$ as $i\to\infty$;
\item $\Orb_i(t)$ is bounded. In this case, there is no
fixed vertices starting some level, by 3) above.
\end{enumerate}

Finally, we will need the following statement from the Galois theory
(see, e.g. \cite[Sect. 3.5]{DixonMortimer}):
\begin{lem}\label{lem:galois}
A solvable transitive subgroup of the symmetric group $S_p$,
where $p$ is prime, is isomorphic
either to $\Z_p$, or to $\Z_p \rtimes \Z_{p-1}$. 
In particular, it is either abelian, or contains an odd
permutation (a generator of $\Z_{p-1}$).
\end{lem}

\section{Finite order automorphisms and around}\label{sec:finordaut}

\begin{lem}\label{lem:key_b}
Let $\phi:G\to G$ be an automorphism of oder $n<\infty$
of a weakly branch group
with 
$\phi(g)=t g t^{-1}$. Then there exists $j_0$
such that for any $j\ge j_0$ there exists an element
$g_j\in \St_j$ and a number $i>j$ such that $\{g_j\}_\phi\cap
\St_i=\varnothing$.
\end{lem}

\begin{proof}
It is sufficient to find an element $g_j$ such that
$h g_j t h^{-1} t^{-1} \ne e$ at the level $L_i$ for any $h\in G$,
or equivalently
\begin{equation}\label{eq:gthth}
g_j t \ne h^{-1} t h. 
\end{equation}

By the condition, $t$ has at some level an orbit of length $n$,
and does not have a longer orbit.
Take for $j_0$ the first time when $t$ has on $L_{j_0}$ an
orbit of length $n$. Then the orbits of successors also will
have the length $n$. Consider any $j\ge j_0$ and an orbit
of length $n$ in $L_j$. Let $v_0\in L_j$ be a vertex from this
orbit. Using the weak branching property we can find a
non-trivial element
$g_j \in \Rist(v_0)$. Let $i$ be the first level where
$g_j$ acts non trivially, say at $v\in \mathcal T_{v_0}$ (see 
Fig. \ref{fig:one}).
\begin{figure}[ht]
\begin{pspicture}(12,10)
\psellipse(6.0,1.0)(6,1)
\psellipse(6.0,1.0)(5.5,0.9)
\psellipse(6.0,1.0)(5,0.8)
\psellipse(6.0,1.0)(4.5,0.7)
\psellipse(6.0,1.0)(4,0.6)
\psellipse(6.0,1.0)(3.5,0.5)
\psellipse(6.0,1.0)(3,0.4)
\psellipse(6.0,1.0)(2.5,0.3)
\psellipse(6.0,1.0)(2,0.2)
\psdot[dotsize=2mm](0,1)
\psdot[dotsize=2mm](0.5,1)
\psdot[dotsize=2mm](1,1)
\psdot[dotsize=2mm](1.5,1)
\psdot[dotsize=2mm](2,1)
\psdot[dotsize=2mm](2.5,1)
\psdot[dotsize=2mm](3,1)
\psdot[dotsize=2mm](3.5,1)
\psdot[dotsize=2mm](4,1)
\psdot[dotsize=2mm](0.5,5)
\psdot[dotsize=2mm](2.,5)
\psdot[dotsize=2mm](3.5,5)
\psdot[dotsize=2mm](2.0,9)
\psellipse(6.0,9.0)(4,0.6)
\psellipse(6.0,5.0)(5.5,0.9)
\psellipse(6.0,5.0)(4,0.6)
\psellipse(6.0,5.0)(2.5,0.3)
\psline[linewidth=2.pt](2,9)(0.5,5)
\psline[linewidth=2.pt](2,9)(2,5)
\psline[linewidth=2.pt](2,9)(3.5,5)
\psline[linewidth=2.pt](0,1)(0.5,5)
\psline[linewidth=2.pt](0.5,1)(0.5,5)
\psline[linewidth=2.pt](1,1)(0.5,5)
\psline[linewidth=2.pt](1.5,1)(2,5)
\psline[linewidth=2.pt](2,1)(2,5)
\psline[linewidth=2.pt](2.5,1)(2,5)
\psline[linewidth=2.pt](3,1)(3.5,5)
\psline[linewidth=2.pt](3.5,1)(3.5,5)
\psline[linewidth=2.pt](4,1)(3.5,5)
\psarc[linewidth=2.5pt](0.5,1){0.5}{180}{360}
\psline*(1,1)(1.1,0.7)(0.75,0.75)
\psarc[linewidth=2.5pt](3.5,1){0.5}{180}{360}
\psline*(3,1)(3.25,0.75)(2.9,0.7)
\put(1.5,9){$v_0$}
\put(0,0.2){$g_j$}
\put(-0.3,1.2){$v$}
\put(12.1,1.){$L_i$}
\put(12.1,5.){$L_{i-1}$}
\put(12.1,9.){$L_j$}
\put(9,7){$t$-orbits}
\psline[linewidth=.5pt]{->}(9.2,7.5)(9,8.5)
\psline[linewidth=.5pt]{->}(9.2,6.8)(8.5,5.2)
\psline[linewidth=.5pt]{->}(9.4,6.8)(9.6,5.4)
\psline[linewidth=.5pt]{->}(9.6,6.8)(10.5,5.6)
\end{pspicture}
\caption{}\label{fig:one}
\end{figure}     
Then 
\begin{equation}\label{eq:gtn}
(g_j t)^n(v)=g_j t^n(v)=g_j(v)\ne v,
\end{equation}
because the $t$-orbit of $v$ has the form $v, t(v), \dots, t^{n-1}(v)$,
$t^n(v)=v$ and $t(v), \dots, t^{n-1}(v)\not\in \mathcal{T}_{v_0}$
implying $gt(v)=t(v), \dots, g t^{n-1}(v)=t^{n-1}(v)$.
So, $g_jt$ has an orbit of length $>n$ and can not be conjugate
to $t$ at the level $L_i$. We obtain (\ref{eq:gthth}).
\end{proof}

\begin{teo}\label{teo:finordaut}
Suppose, $G$ is a saturated weakly branch group
and each automorphism from $\Out (G)$ is of finite order.
Then $G$ has the $R_\infty$ property.
\end{teo}

\begin{proof}
By Lemma \ref{lem:ReidInnerShifts} it is sufficient to
verify $R(\phi)=\infty$ for some $\phi$ of finite order $n$.

By Theorem \ref{teo:LavrNekr} $\phi(g)= tgt^{-1}$ for an automorphism
$t$ of the tree. Then Lemma \ref{lem:key_b} gives inductively
an infinite sequence of representatives of distinct Reidemeister
classes. Thus $R(\phi)=\infty$.
\end{proof}

\begin{ex}\label{ex:GrigorGuptaS}
{\rm
The most studied branch groups -- the Grigorchuk group 
\cite{GrFA} and
the Gupta-Sidki group \cite{GuSi} -- have outer automorphisms of finite order
\cite{GrigSidki,SidkiJA2}.
}
\end{ex}

\begin{ex}\label{ex:allauto}
{\rm
A more evident example is the group of all isometries 
of a symmetric rooted tree. In this case all automorphisms
are inner.
}
\end{ex}

\section{Finite number of orbits}\label{sec:finnumorb}

Now we consider the opposite case, when the number of orbits
$t$ on $L_i$ is uniformly bounded. We will need to restrict ourselves to the WST case.

\begin{lem}\label{lem:key_a}
Let $\phi:G\to G$ be an automorphism 
of a {\rm WST} group
with 
$\phi(g)=t g t^{-1}$, where $t$ is an automorphism of the
tree $\cT$. Suppose, $t$ satisfies
{\rm (b)} above, namely, $\max_i \Orb_i(t)=M<\infty $. Then there exists $j_0$
such that for any $j\ge j_0$ there exists an element
$g_j\in \St_j$ and a number $i>j$ such that $\{g_j\}_\phi\cap
\St_i=\varnothing$.
\end{lem}

\begin{proof}
Let $j_0$ be the level of stabilization of the number of orbits,
i.e., $\Orb_{j_0-1}(t)<M$ and
$\Orb_{j_0}(t)=M$, hence $\Orb_{j}(t)=M$ for any $j\ge j_0$. 
Note that the lengths of orbits of $t$ at next levels, are the
multiples of lengths of orbits on $L_{j_0}$ (with the coefficient
equal to the appropriate product of branching numbers) and
an orbit of smallest length (not unique generally) lies under a
smallest orbit on $L_{j_0}$. 

Now take an arbitrary $j\ge j_0$ and consider an orbit of $t$ of the
smallest size on $L_j$. Let $v$ be a vertex from this orbit,
and find by the WST property an element
$v_0\in\mathcal{T}_v$,
$v_0\in L_{i-1}$ for some $i$, 
with a transitive action of $G_{\{v_0\}}$
on its immediate successors. Let $v_1$ be one of these successors.
Then, as it was explained, its $t$-orbit has the smallest length
among the orbits on $L_i$. This length is equal to $m\cdot b$,
where $m$ is the length of $t$-orbit of $v_0$ and $b$ is the branching
number of $v_0$. Choose $g_j \in G_{\{v_0\}}$ such that
$g_j t^m (v_1)=v_1$ (see 
Fig. \ref{fig:two}).
\begin{figure}[ht]
\begin{pspicture}(12,10)
\psellipticarc(6.0,1.0)(6,1){0}{183}
\psellipticarc(6.75,1.0)(5.25,0.85){180}{360}
\psellipticarc(6.0,1.0)(4.5,0.7){0}{180}
\psellipticarc(6.75,1.0)(3.75,0.55){180}{360}
\psellipticarc(6.0,1.0)(3,0.4){0}{180}
\psellipticarc(5.,1.0)(4.,0.5){178}{360}
\psdot[dotsize=2mm](0,1)
\psdot[dotsize=2mm](0.5,1)
\psdot[dotsize=2mm](1,1)
\psdot[dotsize=2mm](1.5,1)
\psdot[dotsize=2mm](2,1)
\psdot[dotsize=2mm](2.5,1)
\psdot[dotsize=2mm](3,1)
\psdot[dotsize=2mm](3.5,1)
\psdot[dotsize=2mm](4,1)
\psdot[dotsize=2mm](0.5,5)
\psdot[dotsize=2mm](2.,5)
\psdot[dotsize=2mm](3.5,5)
\psdot[dotsize=2mm](2.0,9)
\psellipse(6.0,9.0)(4,0.6)
\psellipticarc(6.0,5.0)(5.5,0.9){0}{180}
\psellipticarc(6.75,5.0)(4.75,0.75){180}{360}
\psellipticarc(6.0,5.0)(4,0.6){0}{180}
\psellipticarc(6.75,5.0)(3.25,0.5){180}{360}
\psellipticarc(6.0,5.0)(2.5,0.3){0}{180}
\psellipticarc(4.5,5.0)(4,0.5){180}{360}
\psline[linewidth=2.pt](2,9)(0.5,5)
\psline[linewidth=2.pt](2,9)(2,5)
\psline[linewidth=2.pt](2,9)(3.5,5)
\psline[linewidth=2.pt](0,1)(0.5,5)
\psline[linewidth=2.pt](0.5,1)(0.5,5)
\psline[linewidth=2.pt](1,1)(0.5,5)
\psline[linewidth=2.pt](1.5,1)(2,5)
\psline[linewidth=2.pt](2,1)(2,5)
\psline[linewidth=2.pt](2.5,1)(2,5)
\psline[linewidth=2.pt](3,1)(3.5,5)
\psline[linewidth=2.pt](3.5,1)(3.5,5)
\psline[linewidth=2.pt](4,1)(3.5,5)
\psarc[linewidth=2.5pt](0.5,1){0.5}{180}{360}
\psline*(1,1)(1.1,0.7)(0.75,0.75)
\put(1.5,9){$v$}
\put(0.,5){$v_0$}
\put(0,0.2){$g_j$}
\put(-1.2,1.2){$t^m(v_1)$}
\put(0.9,0.3){$v_1$}
\put(12.1,1.){$L_i$}
\put(12.1,5.){$L_{i-1}$}
\put(12.1,9.){$L_j$}
\put(9,7){$t$-orbits}
\put(9,2){part of $t$-orbit}
\psline[linewidth=.5pt]{->}(9.2,7.5)(9,8.5)
\psline[linewidth=.5pt]{->}(9.6,6.8)(10.5,5.6)
\end{pspicture}
\caption{}\label{fig:two}
\end{figure}      
By the definition of $G_{\{v_0\}}$,
$$
g_jt (v_1)=t(v_1),\quad (g_j t)^2(v_1)=t^2 (v_1),\dots\qquad 
(g_j t)^m (v_1)=g_j t^m (v_1)=v_1.
$$ 
Hence, the smallest length of a $(g_j t)$-orbit on $L_i$
is $m<m\cdot b=$the smallest length of a $t$-orbit on $L_i$.
Thus, $g_i t$ and $t$ can not be conjugate and we arrive to
(\ref{eq:gthth}) and the same argument as in the beginning
of the proof of Lemma \ref{lem:key_b}, completes the proof.
\end{proof}

\begin{rk}
By Lemma \ref{lem:phiorbitsinclasses} $R(\phi)<\infty$, if the
number of $\phi$-orbits is finite. But the number of
$\phi$-orbits in $G$
is rather weakly related to the number of $t$-orbits on $\cT$.
For example, for $t=\Id$, this depends on ``how saturated 
$G$ is''.
\end{rk}

Similarly to the proof of Theorem \ref{teo:finordaut}, one
can deduce from Lemma \ref{lem:key_a} the following statement.

\begin{teo}\label{teo:finnumoforb}
If a saturated weakly branch group $G$ is a WST group and each its
non-trivial
outer automorphism has the properties from Lemma \ref{lem:key_a},
then $G$ is an $R_\infty$ group.
\end{teo}

\begin{ex}
{\rm
We do not expect interesting examples of \emph{groups} here, 
moreover, we need the results of this section mostly as a tool
for proofs with using for some \emph{automorphisms} in the next section
(case b) below).

Nevertheless, Example \ref{ex:allauto} works here too.
}
\end{ex}

\section{The general case}\label{sec:gencase}

\begin{lem}\label{lem:key_prime}
Let $\phi:G\to G$ be an automorphism of a group
$G$ acting level-transitively
on a spherically symmetric tree $\mathcal{T}$, 
with $\phi(g)=t g t^{-1}$. 
Suppose, 
\begin{enumerate}
\item $G$ is finitely generated;
\item $G$ is a WST group;
\item moreover, for an infinite subsequence $\{i_k\}$of the sequence
of levels, arising as transitivity levels in the definition of WST,
the corresponding group $\G_{\{i\}}$ (see Def. \ref{dfn:gammai}) 
is not abelian; 
\item branching numbers are prime (may be distinct for distinct
levels).
\end{enumerate}

Suppose, $R(\phi)<\infty$. Then there exists $j_0$
such that for any $j\ge j_0$ there exists an element
$g_j\in \St_j$ and a number $i>j$ such that $\{g_j\}_\phi\cap
\St_i=\varnothing$.
\end{lem}

\begin{proof}
As in the proof of Lemma \ref{lem:key_b},
it is sufficient to find an element $g_j\in\St_j$ such that
 at the level $L_i$ for any $h\in G$
\begin{equation}\label{eq:gthth1}
g_j  \ne h^{-1} t h t^{-1}. 
\end{equation}

Consider two cases:
\begin{itemize}
\item[a)] $\Orb_i(t)\to \infty$;
\item[b)] $\Orb_i(t)$ is bounded.
\end{itemize}

\textbf{Case a).} Since $R(\phi)<\infty$ and $G$ is finitely generated, by Lemma \ref{lem:jaforfg} the number of $\phi$-fixed elements 
for the quotient $G/St_{i}$ is
strictly less $\Orb_{i-1}(t)$ at each level $i$
greater some $j_0$. 

Now for any $j>j_0$, let $i-1>j$ be the  number of a level 
with transitive action
of $G_{\{v\}}$ for any $v\in L_{i-1}$ (see Remark \ref{rk:Gvforleveltrans}) such that $\G_{\{i\}}$ is not abelian.

Each of the above-mentioned $\phi$-fixed elements
(except of $e$)
acts non-trivially at some vertex $w_s$. Thus an element, which
fixes these vertexes, is not $\phi$-fixed. Hence there exists $v_0\in L_{i-1}$
such that for any $v$ in its $t$-orbit,
$p_{i}(G_{\{v\}})$ does not contain $\phi$-fixed points,
where $p_{i}:G \to  G/St_{i}$. 
Suppose, the $t$-orbit of $v_0$ has some length $k$ ($k=1$ can
occur in particular). 
Then 
$$
p_{i}G_{\{t^m(v_0)\}}=t^m 
p_{i} G_{\{v_0\}} t^{-m}, \qquad
m=0,1,\dots,k-1.
$$ 
Evidently, elements of these groups commute, and 
we can form a group 
$$
\G:=p_{i}(G_{\{v_0\}})\cdots
p_{i}G_{\{t^{k-1}(v_0)\}}
$$ 
with an action of $\phi$.
Each $\g\in \G$ acts trivially on all $w_s$. Hence, $\G$ has
no nontrivial $\phi$-fixed elements.
So, $\G$
is a subgroup with a fixed-point-free automorphism $\phi$.
Then it is solvable by \cite{Rowley1995}.

Hence, its subgroup $p_i(G_{\{v_0\}})$ is also solvable.
It is a transitive subgroup of the symmetric group $S_p$,
where $p$ is the prime branching number for vertexes from $L_{i-1}$.  
Then, by Lemma \ref{lem:galois}, it is either abelian, or contains
an odd permutation $p_i(g_j)\not\in A_p$, $g_j\in G_{\{v_0\}}$. 
In the first case, $\G_{\{i\}}$ is abelian in a contradiction 
with the
supposition. In the second case, $g_j$ is trivial on $L_i$
except the successors of $v_0$. So it is an odd permutation
on the entire $L_i$, while $h^{-1} t h t^{-1}$ is an even one. This gives (\ref{eq:gthth1}).

\textbf{Case b).} This case immediately follows from
Lemma \ref{lem:key_a}. 
\end{proof}

Similarly to the proof of Theorem \ref{teo:finordaut} we obtain
from Lemma \ref{lem:key_prime} the following statement.

\begin{teo}\label{teo:primecase}
Suppose, $G$ is a finitely generated saturated weakly branch WST 
group on a spherically symmetric tree with prime branching numbers 
and an infinite sequence of non-abelian $\G_{\{i\}}$ (i.e. satisfying
the suppositions of Lemma \ref{lem:key_prime}). Then $G$ is an
$R_\infty$ group.
\end{teo}

\begin{rk}
{\rm
Reasonable examples will be given in the next
section for a version of this statement, namely Theorem \ref{teo:gengen}.
}
\end{rk}

\section{Some generalizations}

Evidently the above statements can be easily extended to
some more general cases (with more complicated formulations).

For example, Theorem \ref{teo:finordaut}
can be evidently generalized in the following way.
\begin{teo}\label{teo:finordautgen}
Suppose, $G$ is a weakly branch group
and each automorphism from $\Out (G)$ is of finite order
and defined by an automorphism of the tree.
Then $G$ has the $R_\infty$ property.
\end{teo}

Now we will give another version of Theorem \ref{teo:primecase}.

\begin{teo}\label{teo:gengen}
Suppose, $G$ is a saturated weakly branch
group on a spherically symmetric tree, such that 
for any $j$, $\St_j$ contains an element $g_j$ defining
an odd permutation at some level $j_0>j$.
Then $G$ is an $R_\infty$ group.
\end{teo}

\begin{proof}
Indeed, (\ref{eq:gthth1}) keeps, because $h^{-1} t h t^{-1}$
is an even permutation and $g_j$ is an odd permutation at
the level $j_0$.
\end{proof}

\begin{ex}
{\rm
The full isometry group as in Example \ref{ex:allauto} satisfies
the conditions of Theorem \ref{teo:gengen}.
}
\end{ex}

\begin{ex}
{\rm
Consider a saturated weakly branch group $G$ and consider
a group $\G$ generated by $G$ and an infinite series
of isometries $g_j$,
e.g., transpositions of two neighbouring elements at level 
$L_{j+1}$ and somehow defined at their successors.
Then $\G$ satisfies the conditions of Theorem \ref{teo:gengen}.
}
\end{ex}

\end{document}